\documentclass[12pt]{amsart}
\usepackage{amssymb}
\usepackage{amsfonts}
\usepackage{amssymb,latexsym}
\usepackage{enumerate}
\usepackage{mathrsfs}

\usepackage{graphicx}
\usepackage[all]{xy}

\newtheorem{thm}{Theorem}[section]

\newtheorem{prop}[thm]{Proposition}

\theoremstyle{definition}
\newtheorem{definition}[thm]{Definition}

\theoremstyle{remark}

\begin{document}

\title[{\normalsize {\large {\normalsize {\Large {\LARGE }}}}}angelic way for modular  Lie algebras toward Kim's conjecture
]{Angelic way for modular  Lie algebras  toward Kim's conjecture  }

\author{Kim YangGon }

\address{emeritus professor,
	 Department of Mathematics,  
	Jeonbuk National University, 567 Baekje-daero, Deokjin-gu, Jeonju-si,
	Jeollabuk-do, 54896, Republic of Korea.}

\

\

\email{ kyk1.chonbuk@hanmail.net }

\subjclass[2010]{Primary00-02,Secondary17B50,17B10}

\begin{abstract}
We consider modular Lie algebras over algebraically closed field of characteristic $p \geq 7.$
This paper purports to prove the conjecture  that  classical modular Lie algebras,in particular  of $C_l$ and of $A_l$ type,  should be a Park's Lie algebra, and so a Hypo- Lie algebra.

\end{abstract}

\maketitle

\section{introduction}

\

\

\large{

If there is  a  Lee's basis  except for a finite number of  simple modules  for a Lie algebra[4], then we would like to say that  the Lie algebra has an  angelic  way.\newline

In this paper we shall see that  modular $C_l$-type and $A_l$- type Lie algebras have angelic  ways.\newline

For this we shall proceed in the following order:
Section 2  deals with modular $A_l$- type Lie algebra and its representation, followed by $C_l$-type Lie algebra and its representation  in section 3.\newline

Finally in section 4 we shall make concluding remarks relating to Park's Lie algebra and Hypo Lie algebra.\newline

We shall assume throughout that  $F$   denotes any algebraically closed field of characteristic $p\geq 7$ unless otherwise stated.

\

\section{modular $A_l$-type Lie algebra  and its representation}

\

\
We must recall first  definitions  related to modular representation theory.

\begin{definition}
Let  $(L , [p])$ be a restricted Lie algebra over $F$  and $\chi \in L^{*}$ be a linear form. If a representation
 $\rho_\chi:$$L\longrightarrow$ $ \mathfrak{gl}(V)$   of  $(L, [p])$ satisfies  $\rho_\chi( x^p- x^{[p]})$= $\chi(x)^p id_V$ 
for any $x\in L,$ then $\rho_\chi$  is said  to be a  $\chi- representation.$\newline

 In  this case we say that the representation or the 
corresponding module has a      $p-character  \chi.$ In particular if $\chi$=0 , then $\rho_0$ is  called  a $ restricted$  representation, whereas  $\rho_\chi$ 
for $\chi \neq0$  is called a $nonrestricted$ representation .
\end{definition}

We are well aware that  we have  $\rho_\chi(a)^p -\rho_\chi ( a^{[p]})$=$\chi(a)^p id_V$ for some $\chi\in L^*$,  for any  $a\in L$  and for any irreducible  representation $\rho_\chi.$ \newline

For an algebraically closed field $F$ of prime characteristic $p,$ the $A_l$- type Lie algebra $L$ over $F$ is just the analogue over $F$ of the $A_l-$ type 
simple Lie algebra over $\mathbb C$.\newline

 In other words, the $A_l$- type Lie algebra over $F$ is isomorphic to the Chevalley Lie algebra of the form 
$\sum_{i=1}^{n}\mathbb{Z}c_i\otimes_\mathbb{Z}F,$ where $ n$= $dim_FL$ and $x_\alpha$= some $c_i$ for each $\alpha\in\Phi$ , $h_{\alpha}$= some $c_{j}$ with  $\alpha$  some base element  of  $\Phi$  for a Chevalley basis 
\{$c_{i}$\} of  the $A_{l}$ - type  Lie algebra over $\mathbb{C}$.\newline

The $A_l$-type Lie algebra  over $\mathbb C$ has its root system $\Phi $= $\{\epsilon_i- \epsilon_j | 1 \leq i \neq j \leq l+1 \}$,where $\epsilon_i$'s are orthonormal unit vectors in the Euclidean space $\mathbb R^{l+1}$.The base of  $\Phi$ is equal to $\{\epsilon_i - \epsilon_{i+1}| 1\leq i \leq l\}$.\newline

We let $L$ be an $A_l$-type simple Lie algebra over an algebraically closed field of  characteristic $p \geq 7$.

For a root $\alpha \in \Phi,$ we put $g_\alpha :=  x_\alpha^{p-1}- x_{- \alpha}$and $w_\alpha:= (h_\alpha+ 1)^2+ 4x_{-\alpha}x_\alpha$.\newline

We have seen from [4],[1] that  any $A_l$-type modular Lie algebra  over $F$ becomes a Park's  Lie algebra.However we would like to  specfy the proof  when $\chi(H) \neq 0$ for a CSA $H$ of $L$ .

 \begin{thm}
Suppose that $\chi$ is  a character  of any  simple $L$-module with $\chi(h_\alpha)\neq 0 $ for some $\alpha \in$ the base of $\Phi$,where $h_\alpha $ is an  element  in  the Chevalley basis of $L$ such that $Fx_\alpha+ Fx_{-\alpha}+F h_\alpha= \frak {sl}_2 (F) $ with   $[x_\alpha,x_{-\alpha}]= h_\alpha \in H $(a CSA of  $L)$.\newline

We then have that  the dimension of any simple $L$-module  with character $\chi= p^m= p^{(n-l) \over 2}$,where   $  n$= dim $L= 2m+ l$ for $ H $ with dim $  H= l$.
\end{thm}

\begin{proof}
If $\chi(x_\alpha)\neq 0$ or $\chi(x_{-\alpha})\neq 0$, then our assertion is evident from [1],[3],[4].So we may assume that $\chi(x_\alpha)= \chi(x_{-\alpha})= 0 $ but $\chi(h_\alpha)\neq 0$.\newline

Furthermore we may put $\alpha= \epsilon_1- \epsilon_2$ without loss of generality since all roots are conjugate under the Weyl group of $\Phi$.\newline

Since the case for $l=1$ is trivial, we may assume $l \geq 2$. For $i=1.2, \cdots$, we put $B_i:= b_{i1}h_{\epsilon_1- \epsilon_2}+\cdots + b_{il}h_{\epsilon_l- \epsilon_{l+1}} $ as in [3],[4] and we put $\frak B:= \{(B_1+ A_{\epsilon_1- \epsilon_2})^{i_1}\otimes (B_2+A_{\epsilon_2 - \epsilon_1})^{i_2}\otimes (\otimes_{j=3}^{l+1}(B_j+ A_{\epsilon_1- \epsilon_j})^{i_j})\otimes(\otimes_{j=3}^{l+1}(B_{l-1+j}+ A_{\epsilon_j- \epsilon_1})^{i_{l-1+j}})\otimes (\otimes_{j=3}^{l+1}(B_{2l-2+j}+ A_{\epsilon_2- \epsilon_j})^ {i_{2l-2+j}})\otimes (\otimes _{j=3}^{l+1}(B_{3l-3+j}+ A_{\epsilon_j- \epsilon_2})^{i_{3l-3+j}}\otimes \cdots \otimes (B_{2m-1}+ A_{\epsilon_l- \epsilon_{l+1}})^{i_{2m-1}}\otimes (B_{2m}+ A_{\epsilon_{l+1}-\epsilon_l})^{i_{2m}}\}$ for $0 \leq i_j \leq p-1$, \newline

where we set \newline

$A_{\epsilon_1- \epsilon_2}= g_\alpha =  g_{\epsilon_1- \epsilon_2}= x_{\epsilon_1- \epsilon_2}^{p-1}- x_{\epsilon_2- \epsilon_1},$ \newline

 $A_{\epsilon_2- \epsilon_1}= c_{\epsilon_2- \epsilon_1}+ (h_\alpha +1)^2 + 4^{-1}x_{-\alpha}x_\alpha,$\newline

$A_{\epsilon_1- \epsilon_3}= g_\alpha^2 (c_{\epsilon_1- \epsilon_3}+ x_{\epsilon_2- \epsilon_3}x_{\epsilon_3- \epsilon_2}\pm x_{\epsilon_1- \epsilon_3}x_{\epsilon_3- \epsilon_1}),$\newline

$ A_{\epsilon_3- \epsilon_1}= g_\alpha^3(c_{\epsilon_3- \epsilon_1}+ x_{\epsilon_3- \epsilon_2}x_{\epsilon_2- \epsilon_3}\pm x_{\epsilon_3- \epsilon_1}x_{\epsilon_1- \epsilon_3})$ or $ x_{\epsilon_3- \epsilon_4}(c_{\epsilon_3- \epsilon_1}+ x_{\epsilon_3- \epsilon_2}x_{\epsilon_2- \epsilon_3}\pm x_{\epsilon_3- \epsilon_1}x_{\epsilon_1- \epsilon_3}), $\newline

$A_{\epsilon_2- \epsilon_j}=  g_\alpha^4( c_{\epsilon_2- \epsilon_3}+ x_{\epsilon_2- \epsilon_3}x_{\epsilon_3- \epsilon_2}\pm x_{\epsilon_1- \epsilon_3}x_{\epsilon_3- \epsilon_1}) $(if $j= 3)$ or  $x_{\epsilon_4- \epsilon_j}(c_{\epsilon_2- \epsilon_j}+ x_{\epsilon_2- \epsilon_j}x_{\epsilon_j- \epsilon_2}\pm x_{\epsilon_1- \epsilon_j}x_{\epsilon_j- \epsilon_1})$,\newline

$A_{\epsilon_j- \epsilon_2}= g_\alpha^5(c_{\epsilon_3- \epsilon_2}+ x_{\epsilon_2-\epsilon_3}x_{\epsilon_3- \epsilon_2} \pm x_{\epsilon_1- \epsilon_3}x_{\epsilon_3- \epsilon_1})$ (if $j= 3)$ or $x_{\epsilon_j- \epsilon_4}(c_{\epsilon_j- \epsilon_2}+ x_{\epsilon_j- \epsilon_2}x_{\epsilon_2- \epsilon_j}\pm x_{\epsilon_j- \epsilon_1}x_{\epsilon_1- \epsilon_j}),$\newline

$A_{\epsilon_2- \epsilon_4}= x_{\epsilon_3- \epsilon_4}^2(c_{\epsilon_2- \epsilon_4}+ x_{\epsilon_2- \epsilon_4}x_{\epsilon_4- \epsilon_2}\pm x_{\epsilon_1- \epsilon_4}x_{\epsilon_4- \epsilon_1}),$\newline

$A_{\epsilon_4- \epsilon_2}= x_{\epsilon_4- \epsilon_3}(c_{\epsilon_4- \epsilon_2}+ x_{\epsilon_4- \epsilon_2}x_{\epsilon_2- \epsilon_4}\pm x_{\epsilon_4- \epsilon_1}x_{\epsilon_1- \epsilon_4}) ,$\newline

$A_{\epsilon_1- \epsilon_j}=x_{\epsilon_3- \epsilon_j}^2(c_{\epsilon_1- \epsilon_j}+ x_{\epsilon_1- \epsilon_j}x_{\epsilon_j- \epsilon_1} \pm x_{\epsilon_2- \epsilon_j}x_{\epsilon_j- \epsilon_2})   , $\newline

$ A_{\epsilon_j- \epsilon_1}    =x_{\epsilon_j- \epsilon_3}^2(c_{\epsilon_J- \epsilon_1}+ x_{\epsilon_1- \epsilon_j}x_{\epsilon_j- \epsilon_1} \pm x_{\epsilon_2- \epsilon_j}x_{\epsilon_j- \epsilon_2}) ,$\newline

$A_{\epsilon_i- \epsilon_j}= x_{\epsilon_i- \epsilon_j}^2$ or $x_{\epsilon_i- \epsilon_j}^3 $ for other roots $\epsilon_i- \epsilon_j$,\newline

where signs are chosen so that they may commute with $x_\alpha$ and $c_\beta$ are chosen so that $A_{\epsilon_2- \epsilon_1}$ and parentheses are invertible in $U(L)/\frak M_\chi$ for the kernel $\frak M_\chi$ in $U(L)$ of any given simple representation of $L$ with the character $\chi$.
\newline

We may see without difficulty that $\frak B$ is a linearly independent set in $U(L)$ by virtue of P-B-W theorem.\newline

We shall prove that a nontrivial linearly dependent  equation leads to absurdity. We assume first that we have a dependence equation which is of least degree with respect to $h_{\alpha_j}\in H$ and the number of whose highest  degree terms is also least.\newline

In case it is conjugated by $x_\alpha$, then there arises a nontrivial dependence equation of lower degree than the given one,which contradicts to our assumption.\newline

Otherwise we have to prove that \newline

(i)$x_{\epsilon_l- \epsilon_k}K+ K'\in \frak M_\chi$ with $l,k \neq 1,2$\newline

(ii) $g_\alpha K+ K'\in \frak M_\chi$ \newline

lead to a contradiction, where both $K$ and $K'$ commute with $x_{\pm \alpha}$  modulo $\frak M_\chi $. In particular $K$ commute with $g_\alpha$.
\newline

For the case (i), we may change it to the form $x_{\alpha}K+ K''\in \frak M_\chi$ for some $K''$ commuting with $x_\alpha= x_{\epsilon_1- \epsilon_2}$ modulo $\frak M_\chi$.\newline

So we have $x_\alpha^p K+ x_\alpha^{p-1}K''\equiv 0$, thus $x_\alpha^{p-1}K''\equiv 0$.\newline

Subtracting from this $x_{-\alpha}x_\alpha K+ x_{-\alpha}K''\equiv 0$, we get \newline

$-x_{-\alpha}x_\alpha K+ g_\alpha K'' \equiv 0$. Recall here that $g_\alpha$ is invertible and $w_\alpha$ belongs to the center of $U(\frak {sl}_2 (F))$ according to [7].\newline

So we get  $4^{-1}\{(h_\alpha +1)^2- w_\alpha\}K+ g_\alpha K''\equiv 0$, and hence\newline

 $(\ast) g_\alpha^{p-1} 4^{-1}\{(h_\alpha + 1)^2- w_\alpha \}K+ cK'' \equiv 0$\newline

 is obtained and from the start equation we have \newline

$(\ast \ast)cx_\alpha K+ c K''\equiv 0$, where $g_\alpha^p- c \equiv 0$.\newline

Subtracting $(\ast \ast)$ from $(\ast)$, we have $4^{-1}g_\alpha^{p-1}\{(h_\alpha+ 1)^2- w_\alpha\}K- cx_\alpha K \equiv 0$.\newline

Multiplying this equation by $g_\alpha^{1-p}$ to the right, we obtain $4^{-1}g_\alpha^{p-1}\{(h_\alpha+ 1)^2- w_\alpha\}g_\alpha^{1-p}K- cx_\alpha g_\alpha^{1-p}K \equiv 0$ \newline

We thus have $4^{-1}\{(h_\alpha+ 1- 2)^2- w_\alpha\}K- x_\alpha g_\alpha K \equiv 0$.

So it follows that $4^{-1}\{(h_\alpha -1)^2- w_\alpha\}K+ x_\alpha x_{-\alpha}K \equiv 0 $.\newline

Next multiplying $x_{-\alpha}^{p-1}$ to the right of this last equation, we obtain $\{(h_{\alpha}- 1)^2- w_\alpha\}K x_{-\alpha}^{p-1}\equiv 0$.
Now multiply $x_\alpha$ in turn consecutively to the left of this equation until it becomes of the form \newline

( a nonzero polynomial of degree $\geq 1$ with respect to $h_\alpha)K $\newline
$\in \frak M_\chi$. \newline

By making use of  conjugation and subtraction consecutively, we are led to a contradiction.$K \in \frak M_\chi$.
\newline

Finally for the case (ii),we consider $K+ g_\alpha^{-1}K' \in \frak M_\chi$.  So we have $x_\alpha K+ x_\alpha g_\alpha^{-1} K' \equiv 0$ modulo $\frak M_\chi$.

By analogy with the argument  as in the case (i), we obtain a contraiction $K \in \frak M_\chi$.

\end{proof}

\

\section{modular  $C_l$-type Lie algebra and its representation }                                            

\

\

We  note first that the root system of $C_l$-type Lie algebra over  $\mathbb C$ is just  $\Phi= \{\pm2 \epsilon_i,\pm (\epsilon_i \pm \epsilon_j)|1\leq i \neq j \leq l \geq 3\}$ with a base $\{\epsilon_1- \epsilon_2,\cdots ,\epsilon_{l-1}- \epsilon_l,2\epsilon_l \}, $ where $\epsilon_i$ and $\epsilon_j$ are linearly independent orthonormal unit vectors in $\mathbb R^l$. \newline

For a root $\alpha \in \Phi,$ we  also put $g_\alpha :=  x_\alpha^{p-1}- x_{- \alpha}$and $w_\alpha:= (h_\alpha+ 1)^2+ 4x_{-\alpha}x_\alpha$  as in section 2, where $[x_\alpha,x_{-\alpha}]= h_\alpha$.\newline

For an algebraically closed field $F$ of prime characteristic $p,$ the $C_l-$ type Lie algebra $L$ over $F$ is just the analogue over $F$ of the $C_l-$ type 
simple Lie algeba over $\mathbb C$.\newline

 In other words the $C_l-$ type Lie algebra over $F$ is isomorphic to the Chevalley Lie algebra of the form 
$\sum_{i=1}^{n}\mathbb{Z}c_i\otimes_\mathbb{Z}F,$ 
where $ n$= $dim_FL$ and $x_\alpha$= some $c_i$ for each $\alpha\in\Phi$ , $h_{\alpha}$= some $c_{j}$ with  $\alpha$  some base element  of  $\Phi$  for a Chevalley basis 
\{$c_{i}$\} of  the $C_{l}$ - type  Lie algebra over $\mathbb{C}$.\newline

We shall compute in this section the dimension of some simple modules of the $C_l$-type Lie algebra $L$ with a CSA  $H$ over an algebraically closed field $F$ of characteristic $p \geq 7$.\newline

Let $L$ be a  $C_l$-type simlpe Lie algebra over an algebraically closed field $F$ of  characteristic $p\geq 7$. Let $\chi$ be a character of any simple $L$-module with $\chi (x_\alpha)\neq 0$ for some $\alpha \in \Phi$,where $x_\alpha$ is an element in the Chevalley basis of  $L$ such that $F x_\alpha + Fh_\alpha+ Fx_{-\alpha}= \frak {sl}_2(F)$ with $[x_\alpha,x_{-\alpha}]= h_\alpha$.\newline

Then we have conjectured in [4]  that any simple $L$-module with character $\chi$ is of dimension  $p^m=p^{n-l\over 2}$,where $n= dim L= 2m +l $ for a CSA  $ H $ with $dim H =l$. \newline

In this section we intend to clarify this conjecture for modular $C_l$-type Lie algebra   $L$. \newline

\begin{prop}\label {thm3.1}

Let  $\alpha$  be any  root  in the root system  $\Phi$ of $L .$ If $\chi(x_\alpha)$ $\neq0,$ then $dim_F$$ \rho_\chi$$(U(L))$ = $p^{2m},$ where $ [Q(U(L)):Q(\mathfrak{Z})]$=$p^{2m}$=$p^{n-l}$ with $\mathfrak{Z}$ the center of $U(L)$  and  $Q$  denotes  the quotient algebra. \newline

So the  simple module corresponding to this representation has $p^m$ as its  dimension.\newline

\end{prop}

\begin{proof}

Let $ \mathfrak {M}_\chi$ be the kernel of this irreducible representation,i.e., a certain (2-sided) maximal ideal of $U(L).$ \newline

(I) Assume first that $\alpha$ is a  short root; then we may put $\alpha$=$\epsilon_1-\epsilon_2$  without loss of generaity since all roots of a given length 
are conjugate under the Weyl group of the root system $\Phi$.\newline

 First we let \newline

 $ B_i$:=$b_{i1}$ $h_{\epsilon_{1}- \epsilon_{ 2}}$+ $b_{i2}$ $h_{\epsilon_{2}-\epsilon_{3}}$+$\cdots$+$b_{i,l-1}$ $h_{\epsilon_{l-1}-\epsilon_{ l }}$ + $b_{il}$ $h_{\epsilon_{l}}$  for $i=  1,2,$ $\cdots,2m,$ where ($b_{i1}$,$b_{i2}$   $\cdots,$$b_{il}$) $\in F^{l}$ are chosen so that any $ (l+1)-$$ B_i$'s  are linearly independent in $\mathbb{P}^l(F),$ the  $\frak B$   below  becomes an $F-$ linearly independent  set in $U(L)$ if necessary and $x_\alpha$$B_i$ $\not\equiv$$B_i$$x_\alpha$ for $\alpha$=$\epsilon_1-$$\epsilon_2.$ \newline

In  $U(L)$/$\mathfrak{M}_\chi$ we claim that we have a basis \newline

$\frak B$:= 
$\{(B_1 +A_{\epsilon_{1}-\epsilon_ { 2}})^{i_1}\otimes(B_2 +A_{-(\epsilon_{1}-\epsilon_ {2})})^{i_2}\otimes \cdots \otimes(B_{2l-2} +A_{-(\epsilon_{l-1}-\epsilon_{l})})^{i_{2l-2}}\otimes(B_{2l-1}+ A_{2\epsilon_{l}})^{i_{2l-1}}\otimes(B_{2l}+A_{-2\epsilon_{l}})^{i_{2l}}\otimes(\otimes_{j=2l+1}^{2m}(B_j+A_{\alpha_{j}})^{i_{j}}) | 0 \leq i_{j}\leq p-1 \}$, \newline

where we put\newline

$A_{\epsilon_{1}-\epsilon_{2}}$= $x_\alpha$=
$x_{\epsilon_{1}-\epsilon_{2}},$ \newline

$A_{\epsilon_{2}-\epsilon_{1}}$=$c_{-(\epsilon_{1}-\epsilon_{2})}$ +$(h_{\epsilon_{1}-\epsilon_{2}}+1)^2$ +4$x_\alpha$ $x_{-\alpha},$ \newline

$A_{{\epsilon_{2}}{\pm}\epsilon_{3}}$=$x_{\pm2\epsilon_{3}}$ $(c_{\epsilon_{2}\pm\epsilon_{3}}+ x_{\epsilon_{2}\pm\epsilon_{3}}x_{-(\epsilon_{2}\pm\epsilon_{3})}\pm x_{\epsilon_{1}\pm\epsilon_{3}}x_{-(\epsilon_{1}\pm\epsilon_{3})}),$ \newline

$A_{\epsilon_{1}+\epsilon_{2}}$=$x_{\epsilon_{1}-\epsilon_{2}}^2$ $(c_{\epsilon_{1}+\epsilon_{2}}+3x_{\epsilon_{1}+\epsilon_{2}}x_{-\epsilon_{1}-\epsilon_{2}}\pm 2x_{2\epsilon_{1}}x_{-2\epsilon_{1}}\pm 2 x_{2\epsilon_{2}}x_{-2\epsilon_{2}}),$\newline

 $A_{\epsilon_{2}\pm\epsilon_{k}}$=$x_{\epsilon_{3}\pm\epsilon_{k}}( c_{\epsilon_{2}\pm\epsilon_{k}}+x_{\epsilon_{2}\pm\epsilon_{k}}x_{-(\epsilon_{2}\pm\epsilon_{k})} \pm x_{\epsilon_{1}\pm\epsilon_{k}}x_{-(\epsilon_{1}\pm\epsilon_{k})} ),$ \newline

$A_{2\epsilon_{2}}$=$x_{2\epsilon_{3}}^2 (c_{2\epsilon_{2}}+2x_{2\epsilon_{2}}x_{-2\epsilon_{2}}\pm 3x_{\epsilon_{1}+ \epsilon_{2}}x_{-\epsilon_{1}-\epsilon_{2}} +2x_{2\epsilon_{1}}x_{-2\epsilon_{1}}),$ \newline

$A_{-2\epsilon_{1}}$=
$x_{-2\epsilon_{3}}^2 ( c_{-2\epsilon_{1}}+ 2x_{-2\epsilon_{1}}x_{2\epsilon_{1}}\pm3x_{-\epsilon_ {1}-\epsilon_{2}}x_{\epsilon_{1}+\epsilon_{2}}
\pm2x_{-2\epsilon_{2}}x_{2\epsilon_{2}}), $ \newline

$A_{-(\epsilon_{1}\pm\epsilon_{3})}$= $x_{-(\pm\epsilon_{3})}(c_{-(\epsilon_{2}\pm\epsilon_{3})}
+  x_{\epsilon_{2}\pm\epsilon_{3}}x_{-(\epsilon_{2}\pm\epsilon_{3})}\pm x_{\epsilon_{1}\pm\epsilon_{3}}x_{-(\epsilon_{1}\pm\epsilon_{3})}),$\newline

$A_{-(\epsilon_{1}\pm\epsilon_{k})}$= $x_{-(\epsilon_{3}\pm\epsilon_{k})}(c_{-(\epsilon_{1}\pm\epsilon_{k})}+ x_{\epsilon_{2}\pm\epsilon_{k}}x_{-(\epsilon_{2}\pm\epsilon_{k})}\pm x_{\epsilon_{1}\pm\epsilon_{k}}x_{-(\epsilon_{1}\pm\epsilon_{k})} ),$ \newline

$A_{2\epsilon_{l}}$= $x_{2\epsilon_{l}}^2, $ \newline

$A_{-2\epsilon_{l}}$= $x_{-2\epsilon_{l}}^2, $ 
\newline

with the sign chosen so that they commute with $x_{\alpha}$ and with $c_{\alpha}\in F$ 
chosen so that $A_{\epsilon_{2}-\epsilon_{1}}$ and parentheses are invertible.  For any other root $ \beta$ we put $A_{\beta}$= $x_{\beta}^2 $ or $x_{\beta}^3 $ if possible. Otherwise attach to these sorts the parentheses(        ) used for designating $A_{-\beta}$ so that  $A_\gamma  \forall \gamma \in \Phi$ may commute with $x_\alpha$.\newline

We shall prove that $\frak B$ is a basis in $U(L))$/$\mathfrak {M}_\chi$. \newline

By virtue of P-B-W theorem, it is not difficult to see that $\frak B$ is evidently a linearly independent set  over $F$ in $U(L)$. Furthermore $\forall$ $\beta$ $ \in\Phi$, $A_{\beta}\notin\mathfrak {M}_\chi$(see detailed proof below).\newline

We shall prove that a nontrivial linearly dependent equation leads to absurdity. We assume first that there is a dependence equation which is of least degree with respect to $h_{\alpha_{j}}\in H$ and the number of whose highest degree terms is also least. \newline

In case it is conjugated by $x_{\alpha}$, then there arises a nontrivial dependence equation of lower degree than the given one, which contradicts to our assumption.\newline

  Otherwise it reduces to one of the following forms:\newline

(i) $x_{2\epsilon_{j}}$$K$ + $K'$ $\in$ $\mathfrak{M}_\chi$ ,\newline

(ii) $x_{-2\epsilon_{j}}$$K$+ $K'$ $\in$ $\mathfrak{M}_\chi$ ,\newline

(iii)$x_{\epsilon_{j}+\epsilon_{k}}$ $K$ + $K'$$\in$ $\mathfrak{M}_\chi$,\newline

(iv)$x_{-\epsilon{j}-\epsilon_{k}}$$K$ + $K'$ $\in$ $\mathfrak{M}_\chi$,\newline

(v)$x_{\epsilon_{j}-\epsilon_{k}}$$K$ + $K'$ $\in$ $\mathfrak{M}_\chi$ ,\newline

where $K$, $K'$ commute with $x_{\alpha}$.\newline

For the case (i), we deduce successively \newline

$x_{\epsilon_{2}-\epsilon_{j}} x_{2\epsilon{j}}$$K$ + $x_{\epsilon_{2}-\epsilon{j}}$$K'$ $\in$ $\mathfrak{M}_\chi$

$\Rightarrow$ $x_{\epsilon_{2}+\epsilon_{j}}$$K$ + $x_{2\epsilon_{j}}$ $x_{\epsilon_{2}-\epsilon_{j}}$$K$ + $x_{\epsilon_{2}-\epsilon{j}}$$K'$
 $\in$ $\mathfrak{M}_\chi$ $\Rightarrow$($x_{\epsilon_{1}+\epsilon_{j}}$ or $x_{2\epsilon_{1}}$)$K$ + $x_{2\epsilon_{j}}$($x_{\epsilon_{1}-\epsilon_{j}}$ or $h_{\epsilon_{1}-\epsilon_{2}}$)$K$ +  ($x_{\epsilon_{1}-\epsilon_{j}}$ or $h_{\epsilon_{1}-\epsilon_{2}}$)$K'$ $\in$ $\mathfrak{M}_\chi$ \newline

by $adx_{\epsilon_{1}-\epsilon_{2}}$if $j$$\neq{1}$ or $j$=1 respectively, so that by successive $ adx_{\alpha}$ and rearrangement we get $x_{\epsilon_{1}\pm\epsilon_{j}}$$K+ K''$  $\in$ $\mathfrak{M}_\chi$ for some $K''$ commuting with $x_{\alpha}$ in view of  the start equation. So (i) reduces to (iii),(iv) or (v). \newline

Similarly as in (i)  and by adjoint operations , (ii) reduces to (iii),(iv) or (v). Also (iii),(iv) reduces to the form (v) putting $\epsilon_{j}$= -(-$\epsilon_{j}$), $\epsilon_{k}$= -(-$\epsilon_{k})$. \newline

Hence we have only to consider the case (v).
We consider\newline
 $x_{\epsilon_{k}-\epsilon_{2}}$ $x_{\epsilon_{j}-\epsilon_{k}}$ $K$+ $x_{\epsilon_{k}- \epsilon_{2}}$$K'$ $\in$ $\mathfrak{M}_\chi$ ,
so that ($x_{\epsilon_{j}-\epsilon_{2}}$+ $x_{\epsilon_{j}-\epsilon_{k}}$$x_{\epsilon_{k}-\epsilon_{2}}$)$K$ + $x_{\epsilon_{k}-\epsilon_{2}}$$K'$
$\in$ $\mathfrak{M}_\chi$ for $j,k$$\neq$1,2 . \newline

We thus have $x_{\epsilon_{j}-\epsilon_{2}}$$K$ + ($x_{\epsilon_{j}-\epsilon_{k}}$$x_{\epsilon_{k}-\epsilon_{2}}$ $K$ + $x_{\epsilon_{k}-\epsilon_{2}}$$K'$) $\in$ $\mathfrak{M}_\chi$, so that we may put this last (       )= another $K'$ alike as in the equation   (v).\newline

 Hence we need to show that  $x_{\epsilon_{j}-\epsilon_{2}}$$K$ + $K'$ $\in$ $\mathfrak{M}_\chi$ leads to absurdity.
We consider \newline

$x_{\epsilon_{2}-\epsilon_{j}}$$x_{\epsilon_{j}-\epsilon_{2}}$$K$ + $x_{\epsilon_{2}-\epsilon_{j}}$$K'$ $\in$ $\mathfrak{M}_\chi$
$\Rightarrow$ ($h_{\epsilon_{2}-\epsilon_{j}}+ x_{\epsilon_{j}-\epsilon_{2}}x_{\epsilon_{2}-\epsilon_{j}})K+  x_{\epsilon_{2}-\epsilon_{j}}K' \in \mathfrak{M}_\chi$ $\Rightarrow $    ($x_{\epsilon_{1}-\epsilon_{2}}$$\pm$$x_{\epsilon_{j}-\epsilon_{2}}$$x_{\epsilon_{1}-\epsilon_{j}}$)$K$ + $x_{\epsilon_{1}-\epsilon_{j}}$ $K'$$\in$ $\mathfrak{M}_\chi$  by $adx_{\epsilon_{1}-\epsilon_{2}}$ $\Rightarrow$ either $x_{\epsilon_{1}-\epsilon_{2}}$$K$ $\in$ $ \mathfrak{M}_\chi$ or  ( $x_{\epsilon_{1}-\epsilon_{2}}$ + $x_{\epsilon_{j}-\epsilon_{2}}$$x_{\epsilon_{1}-\epsilon_{j}}$)$K$+ $x_{\epsilon_{1}-\epsilon_{j}}$$K'$ $\in$ $\mathfrak{M}_\chi$ \newline

depending on [$x_{\epsilon_{j}-\epsilon_{2}}$, $x_{\epsilon_{1}-\epsilon_{j}}$]= +$x_{\epsilon_{1}-\epsilon_{2}}$  or    -$x_{\epsilon_{1}-\epsilon_{2}}$. The former case leads to $K$ $\in$ $\mathfrak{M}_\chi$, a contradiction. \newline

For the latter case

we consider \newline
$x_{\epsilon_{1}-\epsilon_{2}}$$K$ + ( $x_{\epsilon_{j}-\epsilon_{2}}$$x_{\epsilon_{1}-\epsilon_{j}}K$ + $x_{\epsilon_{1}-\epsilon_{j}}$$K'$)

$\in$ $\mathfrak{M}_\chi$.\newline

So we may put\newline

   $(\ast) x_{\epsilon_{1}-\epsilon_{2}}$$K$ + $K''$ $\in$$ \mathfrak{M}_\chi$,\newline

where $K''$=$x_{\epsilon_{j}-\epsilon_{2}}$$x_{\epsilon_{1}-\epsilon_{j}}$$K$+ $x_{\epsilon_{1}-\epsilon_{j}}$$K'$.
Thus $x_{\epsilon_{2}-\epsilon_{1}}$$x_{\epsilon_{1}-\epsilon_{2}}$$K$ + $x_{\epsilon_{2}-\epsilon_{1}}$$K''$ $\in$ $\mathfrak{M}_\chi$.
From $w_{\epsilon_1- \epsilon_2}:=(h_{\epsilon_1- \epsilon_2}+ 1)^2 +4 x_{\epsilon_2- \epsilon_1}x_{\epsilon_1- \epsilon_2}\in $ the center of  $U(\frak{sl}_2(F)), $ we get $ 4^{-1}\{w_{\epsilon_1- \epsilon_2}- (h+ 1)^2\}K+ x_{\epsilon_2- \epsilon_1}K'' \equiv 0 $  modulo $\frak M_\chi$.\newline

If $x_{\epsilon_2- \epsilon_1}^p\equiv c $ which is a constant,then \newline

$(\ast \ast)4^{-1}x_{\epsilon_2-\epsilon_1}^{p-1}\{w_{\epsilon_1- \epsilon_2}- (h_{\epsilon_1- \epsilon_2}+ 1)^2\}K+ cK''\equiv 0 $\newline

 is obtained.
\newline
From $(\ast),(\ast \ast)$, we have\newline

 $4^{-1}x_{\epsilon_2- \epsilon_1}^{p-1}\{w_{\epsilon_1- \epsilon_2}- (h_{\epsilon_1- \epsilon_2}+ 1)^2\}K- cx_{\epsilon_1- \epsilon_2}K\newline
\equiv 0$ modulo $\frak M_\chi$.\newline

Multiplying $x_{\epsilon_1- \epsilon_2}^{p-1}$ to this equation,we obtain \newline

 $(\ast \ast \ast)4^{-1}x_{\epsilon_1- \epsilon_2}^{p-1}x_{\epsilon_2- \epsilon_1}^{p-1}\{w_{\epsilon_1- \epsilon_2}- (h_{\epsilon_1- \epsilon_2}+ 1)^2\}K- cx_{\epsilon_1- \epsilon_2}^pK\equiv 0.  $\newline

By making use of $w_{\epsilon_1- \epsilon_2}$, we may deduce from $(\ast \ast \ast)$ an equation of the form \newline
( a polynomial of degree $\geq 1$ with respect to  $ h_{\epsilon_1- \epsilon_2})K- cx_{\epsilon_1- \epsilon_2}^pK\equiv 0.   $\newline

Finally if we use conjugation and subtraction consecutively,then we are led to a  contradiction $K\in \frak M_\chi.$\newline

(II)Assume next that $\alpha$ is a long root; then we may put $\alpha=2 \epsilon_{1}$ because all roots of the same length are conjugate under the Weyl group of $\Phi$ .\newline

 Similarly as in (I), we let  $B_{i}$:= the same as in (I) except that this time  $\alpha=2\epsilon_{1}$ instead of $\epsilon_{1}-\epsilon_{2}$ . 
\newline

We claim that we have a basis in $U(L)/\frak M_\chi$ such as\newline

 $\frak B$:= $\{(B_{1}+ A_{2\epsilon_{1}})^{i_{1}}\otimes(B_{2} + A_{-2\epsilon_{1}})^{i_{2}}\otimes(B_{3}+ A_{\epsilon_{1}-\epsilon_{2}})^{i_{3}}\otimes(B_{4}+A_{-(\epsilon_{1}-\epsilon_{2})})^{i_{4}}\otimes\cdots \otimes(B_{2l}+ A_{-(\epsilon_{l-1}-\epsilon_{l})})^{i_{2l}}\otimes(B_{2l+1}+ A_{2\epsilon_{l}})^{i_{2l+ 1}}\otimes(B_{2l+ 2}+ A_{-2\epsilon_{l}})^{i_{2l+ 2}}\otimes(\otimes_{j=2l+ 3}^{2m}(B_{j}+ A_{\alpha_{j}})^{i_{j}}; 0 \leq i_{j}\leq p-1\}$ , \newline

where we put \newline

$A_{2\epsilon_{1}}$= $x_{2\epsilon_{1}}$, \newline

$A_{-2\epsilon_{1}}$= $c_{-2\epsilon_{1}}$+ $(h_{2\epsilon_{1}}+ 1)^{2}+ 4x_{-2\epsilon_{1}}x_{2\epsilon_{1}} $, \newline

$A_{-\epsilon_{1}\pm\epsilon_{2}}$=$ x_{-\epsilon_{3}\pm\epsilon_{2}}$$( c_{-\epsilon_{1}\pm\epsilon_{2}}$ $\pm$$x_{-\epsilon_{1}\pm\epsilon_{2}}$$x_{\epsilon_{1}\mp\epsilon_{2}}$$\pm$$x_{\epsilon_{1}\pm\epsilon_{2}}$$x_{-\epsilon_{1}\mp\epsilon_{2}})  $,\newline

$A_{-\epsilon_{!}\pm \epsilon_ {j}}= x_{-\epsilon_{2}\pm\epsilon_{j}}
( c_{-\epsilon_{1}\pm\epsilon_{j}} 
+ x_{\pm{\epsilon_{j}}-\epsilon_{1}}x_{\epsilon_{1}\mp\epsilon_{j}}$
$\pm$$x_{\epsilon_{1}\pm\epsilon_{j}}x_{-\epsilon_{1}\mp\epsilon_{j}})$ ,\newline

 and for any other root $\beta$ we put  $A_{\beta}= x_{\beta}^2 $ or $x_{\beta}^3 $  if possible.\newline

 Otherwise attach to these sorts the parentheses (           ) used for designating $A_{-\beta}$. 
 Likewise as in case (I),  we shall prove that $\frak B$ is a basis in $U(L)$/$\mathfrak{M}_\chi$. \newline

By virtue of P-B-W theorem, it is not difficult to see that $\frak B$ is evidently a linearly independent set over $F$ in $U(L)$. Moreover $\forall\beta$$\in$$ \Phi$, $A_{\beta}\notin \mathfrak{M}_\chi$(see detailed proof below).\newline

We shall prove that a nontrivial linearly dependent equation leads to absurdity. We assume first that there is a dependence equation which is of least degree with respect to $h_{\alpha_{j}}$ $\in$$H$ and the number of  whose highest degree terms is also least.\newline

 If it is conjugated by  $x_{\alpha}$, then there arises a nontrivial dependence equation of least degree than the given one,which contravenes our assumption. \newline

 Otherwise it reduces to one of the following forms: \newline

(i) $x_{2\epsilon_{j}}K + K'\in \mathfrak{M}_\chi$ ,\newline

(ii)  $x_{-2\epsilon_{j}}K + K' \in \mathfrak{M}_\chi$,\newline

(iii)$ x_{\epsilon_{j}+ \epsilon_{k}}K+ K' \in \mathfrak{M}_\chi$ ,\newline

(iv) $ x_{-\epsilon_{j}-\epsilon_{k}}K+ K' \in \mathfrak{M}_\chi$,\newline

(v) $x_{\epsilon_{j}-\epsilon_{k}}K + K' \in \mathfrak{M}_\chi$ ,\newline

where $K$ and $K'$ commute with $x_{\alpha}= x_{2\epsilon_{1}}$.\newline

For the case (i) , we consider  a particular case  $j$=1 first; if we assume $x_{2\epsilon_{1}}K+ K' \in \mathfrak{M}_\chi$, then we are led to a contradiction according to the similar argument ($\ast$) as in (I). \newline

So we assume  $x_{2\epsilon_{j}}K + K' \in \mathfrak{M}_\chi$ with $j\geq2$. Now we have  
$x_{2\epsilon_{j}}K+ K' \in \mathfrak{M}_\chi$ $\Rightarrow$ $x_{-\epsilon_{1}-\epsilon_{j}}x_{2\epsilon_{j}}K + x_{-\epsilon_{1}-\epsilon_{j}}K'\in 
\mathfrak{M}_\chi \Rightarrow $ $x_{-\epsilon_{1}+ \epsilon_{j}}K + x_{2\epsilon_{j}}x_{-\epsilon_{1}-\epsilon_{j}}K + x_{-\epsilon_{1}-\epsilon_{j}}K' \in \mathfrak{M}_\chi \Rightarrow$ by $adx_{2\epsilon_{1}},   x_{\epsilon_{1}+ \epsilon_{j}}K + x_{2\epsilon_{j}}x_{\epsilon_{1}-\epsilon_{j}}K + x_{\epsilon_{1}-\epsilon_{j}}K' \in \mathfrak{M}_\chi$ is obtained. Hence (i) reduces to (iii).\newline

 Similarly (ii)reduces to (iii) or (iv) or (v). So we have only to consider (iii), (iv) , (v). However (iii), (iv), (v) reduce to $x_{2\epsilon_{1}}K + K'' \in \mathfrak{M}_\chi$ after all considering the situation as in (I). Similarly following the argument as in (I), we are led to a contradiction $K \in \mathfrak{M}_\chi$ .

\end{proof}

Now we  are ready to consider another nonzero character $\chi$ different from that of proposition 3.1.

\begin{prop}
Let $\chi $ be a character of any simple $L$-module with $\chi(h_\alpha )\neq $ 0  for  some $\alpha \in \Phi$, where $h_\alpha$ is an element in the Chevalley basis of  $L$ such that $F x_\alpha + Fh_\alpha+ Fx_{-\alpha}= \frak {sl}_2(F)$ with $[x_\alpha,x_{-\alpha}]= h_\alpha \in H$.\newline

 We then have that any simple $L$-module with character $\chi$ is of dimension  $p^m=p^{n-l\over 2}$,where $n= dim L= 2m +l $ for a CSA  H with $dim H =l$.
\end{prop}

\begin{proof}

Let $ \mathfrak {M}_\chi$ be the kernel of this irreducible representation,i.e., a certain (2-sided) maximal ideal of $U(L).$ \newline
If $x_{\epsilon_1- \epsilon_2}\not \equiv 0$ or $x_{\epsilon_2- \epsilon_1}\not \equiv 0$, then our assertion is evident from proposition 4.1 in [3].\newline

So we may let $x_{\epsilon_1- \epsilon_2}\equiv x_{\epsilon_2- \epsilon_1}\equiv 0$ modulo $\frak M_\chi$.\newline

(I) Assume first that $\alpha$ is a  short root; then we may put $\alpha$=$\epsilon_1-\epsilon_2$  without loss of generaity since all roots of a given length 
are conjugate under the Weyl group of the root system $\Phi$.\newline

 First we let  $B_i$:=$b_{i1}$ $h_{\epsilon_{1}- \epsilon_{ 2}}$+ $b_{i2}$ $h_{\epsilon_{2}-\epsilon_{3}}$+$\cdots$+$b_{i,l-1}$ $h_{\epsilon_{l-1}-\epsilon_{ l }}$ + $b_{il}$ $h_{\epsilon_{2l}}$  for $i=  1,2,$ $\cdots,2m,$ where ($b_{i1}$,$b_{i2}$   $\cdots,$$b_{il}$) $\in F^{l}$ are chosen so that any $ (l+1)-$$ B_i$'s  are linearly independent in $\mathbb{P}^l(F),$ the  $\frak B$   below  becomes an $F-$ linearly independent  set in $U(L)$ if necessary and $x_\alpha$$B_i$ $\not\equiv$$B_i$$x_\alpha$ for $\alpha$=$\epsilon_1-$$\epsilon_2.$ \newline

In  $U(L)$/$\mathfrak{M}_\chi$ we claim that we have a basis \newline

$\frak B$:= 
$\{(B_1 +A_{\epsilon_{1}-\epsilon_ { 2}})^{i_1}\otimes(B_2 +A_{-(\epsilon_{1}-\epsilon_ {2})})^{i_2}\otimes \cdots \otimes(B_{2l-2} +A_{-(\epsilon_{l-1}-\epsilon_{l})})^{i_{2l-2}}\otimes(B_{2l-1}+ A_{2\epsilon_{l}})^{i_{2l-1}}\otimes(B_{2l}+A_{-2\epsilon_{l}})^{i_{2l}}\otimes(\otimes_{j=2l+1}^{2m}(B_j+A_{\alpha_{j}})^{i_{j}}) | 0 \leq i_{j}\leq p-1 \}$, \newline

where we put\newline

$A_{\epsilon_{1}-\epsilon_{2}}= g_\alpha=
g_{\epsilon_{1}-\epsilon_{2}},$ \newline

$A_{\epsilon_{2}-\epsilon_{1}}$=$c_{-(\epsilon_{1}-\epsilon_{2})}$ +$(h_{\epsilon_{1}-\epsilon_{2}}+1)^2$ +4$x_{-\alpha}$ $x_\alpha,$ \newline

$A_{{\epsilon_{2}}{\pm}\epsilon_{3}}$=$x_{\pm2\epsilon_{3}}$ $(c_{\epsilon_{2}\pm\epsilon_{3}}+ x_{\epsilon_{2}\pm\epsilon_{3}}x_{-(\epsilon_{2}\pm\epsilon_{3})}\pm x_{\epsilon_{1}\pm\epsilon_{3}}x_{-(\epsilon_{1}\pm\epsilon_{3})}),$ \newline

$A_{\epsilon_{1}+\epsilon_{2}}$=$g_{\epsilon_1-\epsilon_2}^2$ $(c_{\epsilon_{1}+\epsilon_{2}}+2^{-1}x_{\epsilon_{1}+\epsilon_{2}}x_{-\epsilon_{1}-\epsilon_{2}}\pm 3^{-1}x_{2\epsilon_{1}}x_{-2\epsilon_{1}}\pm 3^{-1} x_{2\epsilon_{2}}x_{-2\epsilon_{2}}),$ \newline

$A_{-\epsilon_1- \epsilon_2}= g_{\epsilon_1- \epsilon_2}^3((c_{-\epsilon_{1}-\epsilon_{2}}+2^{-1}x_{\epsilon_{1}+\epsilon_{2}}x_{-\epsilon_{1}-\epsilon_{2}}\pm 3^{-1}x_{2\epsilon_{1}}x_{-2\epsilon_{1}}\pm 3^{-1} x_{2\epsilon_{2}}x_{-2\epsilon_{2}}),$ \newline

$A_{\epsilon_{2}\pm\epsilon_{k}}$=$x_{\epsilon_{3}\pm\epsilon_{k}}( c_{\epsilon_{2}\pm\epsilon_{k}}+x_{\epsilon_{2}\pm\epsilon_{k}}x_{-(\epsilon_{2}\pm\epsilon_{k})} \pm x_{\epsilon_{1}\pm\epsilon_{k}}x_{-(\epsilon_{1}\pm\epsilon_{k})} ),$ \newline

$A_{2\epsilon_{2}}$=$g_{\epsilon_1- \epsilon_2}^6 (c_{2\epsilon_{2}}+3^{-1}x_{2\epsilon_{2}}x_{-2\epsilon_{2}}\pm 2^{-1}x_{\epsilon_{1}+ \epsilon_{2}}x_{-\epsilon_{1}-\epsilon_{2}} +3^{-1}x_{2\epsilon_{1}}x_{-2\epsilon_{1}}),$ \newline

$A_{-2\epsilon_2}= g_\alpha A_{\epsilon_2- \epsilon_1}(c_{-2\epsilon_2}+ 3^{-1}x_{2\epsilon_2}x_{-2\epsilon_2}\pm 2^{-1}x_{\epsilon_1+ \epsilon_2}x_{-\epsilon_1- \epsilon_2}+ 3^{-1}x_{2\epsilon_1}x_{-2\epsilon_1}) $\newline

$A_{2\epsilon_1}= g_{\epsilon_1- \epsilon_2}^4(c_{2\epsilon_1}+ 3^{-1}x_{-2\epsilon_1}x_{2\epsilon_1}\pm 2^{-1}x_{-\epsilon_1- \epsilon_2}x_{\epsilon_1+ \epsilon_2}\pm 3^{-1}x_{-2\epsilon_2}x_{2\epsilon_2})$\newline

$A_{-2\epsilon_1}$=$g_{\epsilon_1- \epsilon_2}^5 ( c_{-2\epsilon_{1}}+ 3^{-1}x_{-2\epsilon_{1}}x_{2\epsilon_{1}}\pm2^{-1}x_{-\epsilon_ {1}-\epsilon_{2}}x_{\epsilon_1+\epsilon_2}
\pm3^{-1}x_{-2\epsilon_2}x_{2\epsilon_2}), $ \newline

$A_{-(\epsilon_{1}\pm\epsilon_{3})}$= $x_{-(\pm\epsilon_{3})}(c_{-(\epsilon_{2}\pm\epsilon_3)}
+  x_{\epsilon_{2}\pm\epsilon_{3}}x_{-(\epsilon_{2}\pm\epsilon_{3})}\pm x_{\epsilon_{1}\pm\epsilon_{3}}x_{-(\epsilon_{1}\pm\epsilon_{3})}),$
\newline

$A_{-(\epsilon_{1}\pm\epsilon_{k})}$= $x_{-(\epsilon_{3}\pm\epsilon_{k})}(c_{-(\epsilon_{1}\pm\epsilon_{k})}+ x_{\epsilon_{2}\pm\epsilon_{k}}x_{-(\epsilon_{2}\pm\epsilon_{k})}\pm x_{\epsilon_{1}\pm\epsilon_{k}}x_{-(\epsilon_{1}\pm\epsilon_{k})} ),$ \newline

$A_{2\epsilon_{l}}$= $x_{2\epsilon_{l}}^2$(if $l\neq 1,2) , $ \newline

$A_{-2\epsilon_{l}}$= $x_{-2\epsilon_{l}}^2, $ 
\newline

with the sign chosen so that they commute with $x_{\alpha}$ and with $c_{\alpha}\in F$ 
chosen so that $A_{\epsilon_{2}-\epsilon_{1}}$ and parentheses are invertible.  For any other root $ \beta$ we put $A_{\beta}$= $x_{\beta}^2 $ or $x_{\beta}^3 $ if possible. \newline

Otherwise attach to these sorts the parentheses(        ) used for designating $A_{-\beta}$ so that  $A_\gamma  \forall \gamma \in \Phi$ may commute with $x_\alpha$.\newline

We shall prove that $\frak B$ is a basis in $U(L))$/$\mathfrak {M}_\chi$. By virtue of P-B-W theorem, it is not difficult to see that $\frak B$ is evidently a linearly independent set  over $F$ in $U(L)$. Furthermore $\forall$ $\beta$ $ \in\Phi$, $A_{\beta}\notin\mathfrak {M}_\chi$(see detailed proof below).\newline

We shall prove that a nontrivial linearly dependent equation leads to absurdity.\newline

We assume first that there is a dependence equation which is of least degree with respect to $h_{\alpha_{j}}\in H$ and the number of whose highest degree terms is also least. \newline

In case it is conjugated by $x_{\alpha}$, then there arises a nontrivial dependence equation of lower degree than the given one, which contradicts our assumption.\newline

 Otherwise it reduces to one of the following forms:\newline

(i)$x_{\pm 2\epsilon_j}K+ K' \in \frak M_\chi, $\newline

(ii)$x_{\pm \epsilon_j \pm \epsilon_k}K+ K' \in \frak M_\chi,$\newline

(iii)$g_{\epsilon_1- \epsilon_2}K+ K' \in \frak M_\chi,$\newline

where $K,K'$ commute with $x_\alpha$ and $x_{-\alpha}$ modulo $\frak M_\chi$.\newline

By making use of proofs of proposition4.1 in [3] and theorem2.1 in [5],we may reduce (i)  and (ii) to the equation of the form\newline

 $x_{\epsilon_1- \epsilon_2}K+ K'\in \frak M_\chi,$\newline

where $K$ commute with $x_{\pm (\epsilon_1- \epsilon_2)}$ and $K'$ commute with $x_{\epsilon_1- \epsilon_2}$ modulo $\frak M_\chi$.\newline

We have $x_{\epsilon_1- \epsilon_2}^p K+ x_{\epsilon_1- \epsilon_2}^{p-1}K' \equiv 0 $, so we get $x_{\epsilon_1- \epsilon_2}^{p-1}K'\equiv 0$.\newline

Subtracting $x_{\epsilon_2- \epsilon_1}x_{\epsilon_1- \epsilon_2}K+ x_{\epsilon_2- \epsilon_1}K'\equiv 0$ from  this equation, we obtain $ -x_{\epsilon_2- \epsilon_1}x_{\epsilon_1- \epsilon_2}K+ g_\alpha K'\equiv 0$.
We should remember that $g_\alpha$ is invertible in $U(L)/\frak M_\chi$ by virtue of  [8].\newline

By the way we use $w_\alpha := (h_\alpha+ 1)^2+ 4x_{-\alpha}x_\alpha \in $ the center of $U(\frak {sl}_2(F))$.Hence we have $-4^{-1}\{w_\alpha- (h_\alpha + 1)^2\}K+ g_\alpha K'\equiv 0$. So we obtain

$4^{-1}g_\alpha^{p-1}\{(h_\alpha+ 1)^2- w_\alpha\}+ cK'\equiv 0 \cdots (\ast)$ \newline

and from the start equation we get \newline

$cx_\alpha K+ cK'\equiv 0 \cdots (\ast \ast)$.\newline

Subtracting $(\ast \ast)$ from $(\ast)$, we get $4^{-1}g_\alpha^{-1}\{(h_\alpha+ 1)^2- w_\alpha\}K- cx_\alpha K \equiv 0$.
Multiplying this equation by $g_\alpha^{1-p}$ to the right, we have\newline

 $4^{-1}g_\alpha^{p-1}\{h_\alpha+1)^2- w_\alpha\}g_\alpha^{1-p}K- cx_\alpha g_\alpha^{1-p}K \equiv 4^{-1}g_\alpha^{p-1} \{(h_\alpha+ 1)^2- w_\alpha\}g_\alpha^{1-p}K+ x_\alpha x_{-\alpha}K \equiv 0$.\newline

Conjugation of the brace of this equation $(p-1)$- times by $g_\alpha$ gives rise to $4^{-1}\{(h_\alpha- 1)^2- w_\alpha\}K+ x_\alpha x_{-\alpha}K\equiv 0$.
Next mutiplying $x_{-\alpha}^{p-1}$ to the right of the last equation, we obtain \newline

$\{(h_\alpha -1)^2- w_\alpha\}K x_{-\alpha}^{p-1}\equiv 0$ modulo $\frak M_\chi$. \newline

Now we multiply $x_\alpha$  to the left of this equaion consecutively until it becomes of the form \newline

(a nonzero polynomial of degree $\geq 1$ with respect to $h_\alpha)K \newline
\equiv 0$ modulo $ \frak M_\chi$.\newline

If we make use of conjugation and subtraction consecutively, then we arrive at a contradiction $K\equiv 0$.\newline

Next for the case (iii),we change it to the form (iii)$'K+ g_\alpha^{-1}K'\in \frak M_\chi$.\newline

We thus have an equation\newline

 $x_{\epsilon_1- \epsilon_2}K+ x_{\epsilon_1- \epsilon_2}g_{\epsilon_1- \epsilon_2}^{-1}K'\equiv 0$ modulo $\frak M_\chi$.
According to the above argument, we are also  led to a contradiction $K\in \frak M_\chi$.

\end{proof}

\

\section{concluding remark}

\

\

We have considered up to now the relationship of  $C_l$ and $A_l$-type modular Lie algebras with Hypo- Lie algebra.\newline

So we may recapitulate the arguments in this paper as follows.

\begin{thm}

Let F be any algebraically closed field of characteristic $p\geq 7$.Let $L$ be any $C_l$ or $A_l$-type modular Lie algebra oveer $F$.We then assert that $L$ is a Park's Lie algebra, and so a Hypo- Lie algebra.

\end{thm}
\begin{proof}

Combining theorem 2.2 , proposition 3.1 and proposition3.2  gives rise to our assertion.

\end{proof}

We are looking forward to claiming that any $B_l$ and $D_l$-type modular Lie algebras also become a Hypo Lie algebra over any algebraically closed field of characteristic $p\geq 7$.\newline

The prime number 7 is important since  all modular  Lie algebras  of classical type are simple  for $p\geq 7$  if we disregard their centers.\newline

Furthermore all modular simple Lie algebras are known to be either of classical type or of Cartan type over any algebraically closed field of characteristic $p\geq 7$.

\

\

\bibliographystyle{amsalpha}

\end{document}